\documentclass[a4paper,twoside]{amsart}
\usepackage{amssymb}
\usepackage{amsmath}
\usepackage{amsthm}
\usepackage[non-compressed-cites]{amsrefs}

\usepackage{mathrsfs}
\usepackage{bbm}

\newcommand{\duline}[1]{{#1}}

\usepackage[normalem]{ulem}

\usepackage{enumerate}

\usepackage{tikz}
\usepackage{float}
\usepackage{pgfmath}
\usepackage{color}

\usepackage{graphicx}
\usetikzlibrary{positioning,patterns,calc}
\usepackage{caption}
\usepackage{subcaption}

\usepackage{comment}

\newcommand{\real}{\mathbb{R}} 
\newcommand{\pinteger}{\mathscr{P}} 
\newcommand{\parts}[1]{\mathbf{\duline{2}}^{#1}} 

\newcommand{\varifolds}{\mathbf{V}} 
\newcommand{\ivarifolds}{\mathbf{IV}} 
\newcommand{\restrict}{\mathop{\llcorner}} 
\newcommand{\setvarifold}[2]{\mathbf{\duline{\nu}}(#1,#2)} 
\newcommand{\graphvarifold}[1]{\mathbf{\duline{\nu}}(#1)} 

\newcommand{\qspace}[2]{\mathbf{\duline{Q}}{}_{#1}(#2)} 

\newcommand{\smoothfunctions}[2]{\mathcal{E}(#1,#2)} 
\newcommand{\smoothcompact}[2]{\mathcal{D}(#1,#2)} 
\newcommand{\compactfunctions}[2]{\mathcal{K}(#1,#2)} 
\newcommand{\vweakfunctions}[2]{\mathcal{T}(#1,#2)} 
\newcommand{\domain}[1]{\text{dmn }#1} 
\newcommand{\weakd}{\mathbf{D}} 
\newcommand{\vweakd}[2]{{#1}\,\weakd{#2}} 
\newcommand{\preimage}[2]{{#1}^{-1}[#2]} 
\newcommand{\notpreimage}[2]{\domain{#1}\setminus{{#1}^{-1}[#2]}} 
\newcommand{\graph}[1]{\textup{graph}(#1)} 
\newcommand{\lip}[1]{\textup{Lip}(#1)} 
\newcommand{\affineap}[1]{\textup{ap}\,A{#1}} 

\newcommand{\openball}[2]{\mathbf{\duline{U}}\left(#1,#2\right)} 
\newcommand{\closedball}[2]{\mathbf{\duline{B}}\left(#1,#2\right)} 
\newcommand{\cbshort}[1]{\mathbf{\duline{B}}\left(#1\right)} 
\newcommand{\grassmannian}[2]{\mathbf{\duline{G}}(#1,#2)} 
\newcommand{\planeopenball}[3]{\mathbf{\duline{U}}\left(#1;#2,#3\right)} 
\newcommand{\planeclosedball}[3]{\mathbf{\duline{B}}\left(#1;#2,#3\right)} 
\newcommand{\tcylinder}[4]{\mathbf{\duline{C}}\left(#1;#2,#3,#4\right)} 
\newcommand{\tcylshort}[2]{\mathbf{\duline{C}}\left(#1,#2\right)} 

\newcommand{\density}[3]{\Theta^{#1}(#2,#3)} 
\newcommand{\gmeancurvature}{\mathbf{\duline{h}}} 
\newcommand{\gsff}[1]{\mathcal{C}_{#1}} 
\newcommand{\totalcurvature}[3]{(\mass{#1}\restrict{#2})_{(#3)}(|\gmeancurvature(#1)|)} 
\newcommand{\smcurvature}[4]{{#3}^\mq\totalcurvature{#1}{\closedball{#2}{#3}}{#4}} 
\newcommand{\tangentmap}[3]{\textup{Tan}^{#1}(#2,#3)} 
\newcommand{\tmap}[1]{\tau_{#1}} 

\newcommand{\mq}{\alpha} 
\newcommand{\econstant}{\Lambda} 
\newcommand{\unitmeasure}[1]{\boldsymbol{\omega}_{#1}} 

\newcommand{\hausdorff}{\mathcal{H}} 
\newcommand{\mass}[1]{\|#1\|} 
\newcommand{\vboundary}[2]{{#1}\,\partial{#2}} 
\newcommand{\support}[1]{\textup{supp}\,#1} 
\newcommand{\lebesgue}[1]{\mathcal{L}^{#1}} 

\newcommand{\dimension}[1]{\textup{dim }#1} 
\newcommand{\integrald}{\textup{d}} 
\newcommand{\integralvx}[2]{\textup{d}\mass{#1}_{#2}} 
\newcommand{\innerproduct}[2]{\langle\, #1,#2\,\rangle} 
\newcommand{\cardinality}[1]{\textup{card }#1} 
\newcommand{\sing}{\textup{sing}} 
\newcommand{\homom}[2]{\mathcal{L}(#1,#2)} 
\newcommand{\trace}[1]{\textup{tr}_{#1}} 

\newcommand{\distance}{\textup{dist}} 
\newcommand{\diam}{\textup{diam}} 
\newcommand{\openballY}[3]{\mathbf{U}^{#1}(#2,#3)} 
\newcommand{\closedballY}[3]{\mathbf{B}^{#1}(#2,#3)} 

\usepackage{etoolbox}
\makeatletter
\patchcmd{\@maketitle}
  {\ifx\@empty\@dedicatory}
  {\ifx\@empty\@date \else {\vskip3ex \centering\footnotesize\@date\par\vskip1ex}\fi
   \ifx\@empty\@dedicatory}
  {}{}
\patchcmd{\@adminfootnotes}
  {\ifx\@empty\@date\else \@footnotetext{\@setdate}\fi}
  {}{}{}
\makeatother

\newtheorem{theorem}{Theorem}[section]
\newtheorem{lemma}[theorem]{Lemma}
\newtheorem{remark}[theorem]{Remark}

\newtheorem{corollary}[theorem]{Corollary}

\newtheorem{definition}[theorem]{Definition}
\newtheorem{hypothesis}[theorem]{Hypothesis}
\newtheorem*{theorem*}{Theorem}

\newtheorem*{corollary*}{Corollary}

\makeatletter
\@addtoreset{claim}{theorem}
\makeatother

\begin{document}
\title[$C^{1,\alpha}$ Multiple Valued Representation]{A Corrected Proof of the Graphical Representation of a Class of Curvature Varifolds by $C^{1,\alpha}$ Multiple Valued Functions}
\author{Nicolau S. Aiex}
\date{\today}
\address{88, Sec.4, Tingzhou Road, SE Building SE809, Taipei, 116059, Taiwan}
\email{nsarquis@math.ntnu.edu.tw}
\maketitle

\begin{abstract}
We provide a counter-example to Hutchinson's original proof of $C^{1,\alpha}$ representation of curvature $m$-varifolds with $L^{q}$-integrable second fundamental form and $q>m$ in \cite{hutchinson1986.2}.
We also provide an alternative proof of the same result and introduce a method of decomposing varifolds into nested components preserving weakly differentiability of a given function.
Furthermore, we prove the structure theorem for curvature varifolds with null second fundamental form which is widely used in the literature.
\end{abstract}

\section{Introduction}

The goal of this article is to address and fix an error in the proof of $C^{1,\alpha}$ multiple valued regularity for curvature $m$-varifolds with $L^q$-integrable ($q>m$) second fundamental form  in \cite{hutchinson1986.2}.
Even though the conclusions therein are correct, there is a minor mistake in the arguments of \cite{hutchinson1986.2}*{Lemma 3.3} to which we provide a counter-example in the Appendix and a substitute result in the form of Corollary \ref{partition corollary}.
Not only does our counter-example creates the need of an alternative proof, but also it indicates how the original argument fails to fully utilize all the necessary hypotheses.
We further point out that the structure theorem \cite{hutchinson1986.2}*{Corollary to Theorem 3.4} as stated is also incorrect, besides, the proof is depedent on \cite{hutchinson1986.2}*{Lemma 3.3}.
We fix the statement and give an alternative proof as well as proving the uniqueness of the tangent cone \cite{hutchinson1986.2}*{Theorem 3.4}.

We would like to further point out that there are many results reliant on the conclusions of \cite{hutchinson1986.2}, particularly on the characterization of curvature varifolds with null second fundamental form.
To cite a few we mention \cites{bellettini-mugnai2004,wickramasekera2005,schatzle2010,mugnai2013,mondino2014,buet-leonardi-masnou2022}.

Besides the main regularity result, the main contribution of our alternative proof is the explicit construction of a decomposition of curvature varifolds that preserves weakly differentiability of a given function and satisfies a nesting property for arbitrarily small scales.
The components of the varifold may not be indecomposable in the sense of \cite{menne2016.1}*{\textsection 6}, that is, the support of each component might have disconnected pieces.
This is the drawback to ensure that the function remains weakly differentiable with respect to each component.
In particular, given a curvature varifold we are able to construct a partition in which each element is also a curvature varifold, though they may be decomposable.
This is given by Theorem \ref{partition theorem} in which the theory of weak differentiable functions developed in \cite{menne2016.1}, in particular the Sobolev-Poincar\'e inequality \cite{menne2016.1}*{10.7}, play an essential role.
We further prove the uniqueness of tangent cone in Lemma \ref{unique tangent cone} and the structure theorem for curvature varifolds with null second fundamental form in Theorem \ref{structure theorem}.

The principal issue of \cite{hutchinson1986.2}*{Lemma 3.3} is to prove that the tangent map of the varifold accummulates on finitely many planes at sufficiently small scale.
Our example in the appendix suggests that the argument should rely strongly on the integrability condition.
We construct a sequence of partitions at sufficiently small scales keeping track of the derived subcomponents.
Even though the number of components of each partition may increase, the Sobolev-Poincar\'e inequality for weakly differentiable functions on varifolds guarantees that the number of components that intersect the half-scale will eventually be constant.

\begin{figure}[H]
\centering
\begin{subfigure}{.5\textwidth}
  \centering
  \begin{tikzpicture}[scale=0.25]
\draw  (0,0) node (v1) {} circle (10);
\draw[blue,very thick,dotted] (-10,0) -- (10,0);

\draw[red,very thick,dotted] (-9.5,3) .. controls (-7,5) and (-6,5.5) .. (-4.5,4.5) .. controls (-2,2.5) and (-2,0) .. (0,0) .. controls (2,0) and (7.5,5.5) .. (6,7) .. controls (5,8) and (3.5,6) .. (2,5) .. controls (0.5,4) and (-1.5,4) .. (-3,5.5) .. controls (-4.5,7) and (-4.5,8) .. (-6,8);

\begin{scope}
\clip (0,0) circle (5);
\draw[blue,very thick] (-10,0) -- (10,0);
\draw[red,very thick] (-9.5,3) .. controls (-7,5) and (-6,5.5) .. (-4.5,4.5) .. controls (-2,2.5) and (-2,0) .. (0,0) .. controls (2,0) and (7.5,5.5) .. (6,7) .. controls (5,8) and (3.5,6) .. (2,5) .. controls (0.5,4) and (-1.5,4) .. (-3,5.5) .. controls (-4.5,7) and (-4.5,8) .. (-6,8);
\end{scope}

\draw[dashed]  (v1) circle (5);

\end{tikzpicture}
  \caption*{Partition at scale $R_1$ and $\frac{R_1}{2}$.}
  \label{fig:sub1}
\end{subfigure}%
\begin{subfigure}{.5\textwidth}
  \centering
  \begin{tikzpicture}[scale=0.25]
\draw  (0,0) node (v1) {} circle (10);
\draw[blue,very thick,dotted] (-10,0) -- (10,0);

\draw[red,very thick,dotted] (-9.5,3) .. controls (-7,5) and (-6,5.5) .. (-4.5,4.5) .. controls (-2,2.5) and (-2,0) .. (0,0) .. controls (2,0) and (7.5,5.5) .. (6,7) .. controls (5,8) and (3.5,6) .. (2,5) .. controls (0.5,4) and (-1.5,4) .. (-3,5.5) .. controls (-4.5,7) and (-4.5,8) .. (-6,8);

\begin{scope}
\clip (0,0) circle (4);
\draw[blue,very thick] (-10,0) -- (10,0);
\draw[red,very thick] (-9.5,3) .. controls (-7,5) and (-6,5.5) .. (-4.5,4.5) .. controls (-2,2.5) and (-2,0) .. (0,0) .. controls (2,0) and (7.5,5.5) .. (6,7) .. controls (5,8) and (3.5,6) .. (2,5) .. controls (0.5,4) and (-1.5,4) .. (-3,5.5) .. controls (-4.5,7) and (-4.5,8) .. (-6,8);
\end{scope}

\draw[dashed]  (v1) circle (4);

\end{tikzpicture}
  \caption*{Partition at scale $R_2<R_1$ and $\frac{R_2}{2}$.}
  \label{fig:sub2}
\end{subfigure}
\label{fig:test}
\end{figure}

After constructing the appropriate decomposition, we proceed to prove the Lipschitz approximation of each individual component.
The integrability condition on the second fundamental form is sufficiently strong to obtain excess estimates via the nesting property of the decomposition and the Sobolev-Poincar\'e inequality.

Once we obtain the Lipschitz approximation by $\mathbf{Q}$-valued function then Hutchinson's proof of $C^{1,\alpha}$ regularity in \cite{hutchinson1986.2}*{3.8,3.9} remains valid since we prove all the auxiliary results derived from \cite{hutchinson1986.2}*{Lemma 3.3} using our alternative conclusion.
We have verified Hutchinson's proof of the H\"older regularity and decided to not include it here since we have nothing new to contribute to that part of the result.

This article is divided as follows. In section 2 we introduce the notion of partition of a varifold and prove all the necessary consequences of the monotocity formula in our setting.
In section 3 we construct the nested partition of a varifold that preserves a weakly differentiable function.
In section 4 we prove the structure theorem Theorem \ref{structure theorem}, the uniqueness of tangent cones Lemma \ref{unique tangent cone} and compute all the estimates to apply a general Lipschitz approximation theorem.
The appendix is dedicated to a counter-example that justifies an alternative proof of Hutchinson's result.

\textbf{
Acknowledgements:} We would like to thank professor Ulrich Menne for several discussions and support inthe development of this project.
The author was funded by NSTC grant 112-2811-M-003-002.

\section{Preliminaries}

\noindent\textbf{Notation.} 
We denote $\openball{x}{r}=\{y\in\real^n:|y-x|<r\}$ and $\closedball{x}{r}$ its closure.
Whenever $U\subset\real^n$ is an open set, $m,n\in\pinteger$ with $m\leq n$, $V\in\varifolds_m(U)$ and $A\subset U$ we denote $(V\restrict A)(B)=V(B\cap A\times\grassmannian{n}{m})$.
Given a finite dimensional vector field $Y$ we denote by $\vweakfunctions{V}{Y}$ the space of generalised $V$-weakly differentiable functions (see \cite{menne2016.1}*{Definition 8.3}).
If $\openball{a}{s}\subset U$ and $f\in\vweakfunctions{V}{Y}$ then we write $V_s=V|{\parts{\openball{a}{s}\times\grassmannian{n}{m}}\in\varifolds_m(\openball{a}{s})}$ and $f_s=f|{\openball{a}{s}}\in\vweakfunctions{V_s}{Y}$.
Given $q,\delta\in\real$ with $q>m$ and $\delta>0$, we define $\mq=1-\frac{m}{q}$and $\econstant=2\unitmeasure{m}^{-\frac{1}{q}}\mq^{-1}$, where $\unitmeasure{m}$ is the volume of the unit $m$-ball.
When $R\subset\real^n$ is a $\hausdorff^m$-rectifiable set and $\theta$ is a $\hausdorff^m$-measurable $\real$-valued positive function defined on $R$, we write $\setvarifold{R}{\theta}$ for the corresponding induced rectifiable $m$-varifold.

\begin{hypothesis}\label{basic hypothesis}
Suppose $m,n\in\pinteger$, $m\leq n$, $U\subset\real^n$ is an open set, $V\in\varifolds_m(U)$, $\mass{\delta V}$ is a Radon measure and $\gmeancurvature(V,\cdot)$ is a $\mass{V}$-measurable $\real^n$-valued function in $U$ satisfy:
\begin{enumerate}[(i)]
\item $\density{m}{\mass{V}}{x}\geq 1$ for $\mass{V}$-almost every $x\in U$ and
\item For all $g\in\compactfunctions{U}{\real^n}$, $\delta V(g)=-\int_U\innerproduct{\gmeancurvature(V,x)}{g(x)}\integralvx{V}{x}$,
\end{enumerate}
where $\compactfunctions{U}{\real^n}$ denotes the space of compactly supported continuous functions.
\end{hypothesis}

It is easy to see that if $V$ satisfies the above hypothesis, $a\in\support{\mass{V}}$ and $r>0$ with $\openball{a}{r}\subset U$ then $V_r$ also satisfies the same conditions.
We also remark that condition $(ii)$ is equivalent to $\mass{\delta V}_{\textup{sing}}=0$ (see \cite{allard1972}*{4.3(5)}).

\begin{definition}[\cite{menne-scharrer2022:arxiv}*{Definition 5.7}]\label{partition definition}
Suppose $m,n\in\pinteger$, $m\leq n$, $U\subset\real^n$ is an open set, $V\in\varifolds_m(U)$ and $\mass{\delta V}$ is a Radon measure.
We say that $\Pi\subset\varifolds_m(U)$ is a partition of $V$ in $U$ if:
\begin{enumerate}
\item For each $W\in\Pi$ there exists a $\mass{V}+\mass{\delta V}$-measurable set $E\subset U$ with $\mass{V}(E)>0$, $\vboundary{V}{E}=0$ and $W=V\restrict E$;
\item For all $\beta\in\compactfunctions{U\times\grassmannian{n}{m}}{\real}$, $V(\beta)=\sum_{W\in\Pi}W(\beta)$ and
\item For all $\theta\in\compactfunctions{U}{\real}$, $\mass{\delta V}(\theta)=\sum_{W\in\Pi}\mass{\delta W}(\theta)$.
\end{enumerate}
Given a varifold $V\in\varifolds_m(U)$, $a\in U$, $r>0$ with $\openball{a}{r}\subset U$, and $\Pi$ a partition of $V_r$, we define $\Pi'=\{W\in\Pi:\support{\mass{W}}\cap\openball{a}{\frac{1}{2} r}\neq\emptyset\}$.
\end{definition}

The following lemma shows that Hypothesis \ref{basic hypothesis} is inherited by elements of a partition of $V$.

\begin{lemma}\label{inherited hypothesis}
Suppose $m,n\in\pinteger$, $U\subset\real^n$ and $V\in\varifolds_m(U)$ satisfy Hypothesis $\ref{basic hypothesis}$.
If $W\in\varifolds_m(U)$ satifies \ref{partition definition}(1) with respect to  $V$ in $U$ then it also satisfies Hypothesis $\ref{basic hypothesis}$.
\end{lemma}
\begin{proof}
Firstly, it follows from \cite{federer1969}*{2.9.11} that $\density{m}{\mass{W}}{x}=\density{m}{\mass{V}}{x}$ for $\mass{W}$-almost every $x\in U$.
Now, let $E\subset U$ be the set that defines $W$.
Since $\vboundary{V}{E}=0$ we have for all $g\in\compactfunctions{U}{\real^n}$
\begin{equation*}
\begin{aligned}
\delta W(g) & = (\delta V\restrict E)(g)=-\int_U\innerproduct{\gmeancurvature(V,x)}{g_E(x)}\integralvx{V}{x}\\
            & = -\int_E\innerproduct{\gmeancurvature(V,x)}{g(x)}\integralvx{V}{x}\\
            & = -\int_U\innerproduct{\gmeancurvature(V,x)}{g(x)}\integralvx{W}{x},
\end{aligned}
\end{equation*}
where $g_E(x)=g(x)$ if $x\in E$ and $0$ otherwise.
This implies that $\gmeancurvature(W,x)=\gmeancurvature(V,x)$ for $\mass{W}$-almost every $x\in U$ and concludes the proof.
\end{proof}

\begin{lemma}[Monotonicity Formula \cite{simon1983:gmtbook}*{17.6,17.7,17.9(2)}]\label{monotonicity formula}
Let $m,n\in\pinteger$, $\delta,q,r\in\real$, $U\subset\real^n$ be an open set, $V\in\varifolds_m(U)$ satisfying Hypothesis \ref{basic hypothesis} and $a\in\support\mass{V}$.
Suppose $m\leq n$, $q>m$, $\econstant\delta\leq 1$, $r>0$ and $\closedball{a}{r}\subset U$.

If $V$ satisfies $\density{m}{\mass{V}}{a}\geq 1$ and $\smcurvature{V}{a}{r}{q}<\delta$ then for all $0<s<t\leq r$:
\begin{equation*}
\begin{aligned}
e^{\econstant\delta\left(\frac{s}{r}\right)^\mq}s^{-m}\mass{V}(\closedball{a}{s}) \leq & e^{\econstant\delta\left(\frac{t}{r}\right)^\mq}t^{-m}\mass{V}(\closedball{a}{t})\\
         & \quad - \int_{\closedball{a}{t}\setminus\closedball{a}{s}\times\grassmannian{n}{m}}\frac{|S^\perp_\sharp(x-a)|^2}{|x-a|^{m+2}}\integrald V_{(x,S)}
\end{aligned}
\end{equation*}
and
\begin{equation*}
\begin{aligned}
e^{-\econstant\delta\left(\frac{s}{r}\right)^\mq}s^{-m}\mass{V}(\closedball{a}{s}) \geq & e^{-\econstant\delta\left(\frac{t}{r}\right)^\mq}t^{-m}\mass{V}(\closedball{a}{t})\\
         & \quad - \int_{\closedball{a}{t}\setminus\closedball{a}{s}\times\grassmannian{n}{m}}\frac{|S^\perp_\sharp(x-a)|^2}{|x-a|^{m+2}}\integrald V_{(x,S)}
\end{aligned}
\end{equation*}
\end{lemma}

\begin{lemma}\label{lower bound mass}
Let $m,n\in\pinteger$, $\delta,q,\gamma\in\real$, $U\subset\real^n$ be an open set, $V\in\varifolds_m(U)$ satisfying Hypothesis \ref{basic hypothesis} and $a\in\support\mass{V}$.
Suppose $m\leq n$, $q>m$, $\econstant\delta<1$ and $0<\gamma<1$.
If $r>0$ with $\closedball{a}{r}\subset U$ and $W\in\varifolds_m(\openball{a}{r})$ satisfies \ref{partition definition}(1) with respect to $V_r\in\varifolds_m(\openball{a}{r})$ are such that
\begin{enumerate}[(i)]
\item $\smcurvature{V}{a}{r}{q}\leq\delta$ and
\item $\support{\mass{W}}\cap\openball{a}{\gamma r}\neq\emptyset$,
\end{enumerate}
then 
\begin{equation*}
r^{-m}\mass{W}(\openball{a}{r})\geq e^{-\econstant\delta(1-\gamma)^\mq}(1-\gamma)^{m}\unitmeasure{m}.
\end{equation*}
\end{lemma}
\begin{proof}
By Lemma \ref{inherited hypothesis} we can find $x\in\support{W}\cap\openball{a}{\gamma r}$ such that $\density{m}{W}{x}\geq 1$.
It follows from Lemma \ref{monotonicity formula} that for all $0<t<s<(1-\gamma)r$ we have
\begin{equation*}
\begin{aligned}
e^{\econstant\delta \left(\frac{t}{r}\right)^{\mq}}t^{-m}\mass{W}(\closedball{x}{t}) & \leq e^{\econstant\delta \left(\frac{s}{r}\right)^{\mq}}s^{-m}\mass{W}(\closedball{x}{s})\\
                                                                       & \leq e^{\econstant\delta (1-\gamma)^{\mq}}s^{-m}\mass{W}(\closedball{a}{\gamma r+ s})
\end{aligned}
\end{equation*}
We conclude the proof by letting $t$ tend to $0$ and $s$ tend to $(1-\gamma)r$.
\end{proof}

\begin{lemma}\label{support of component contains a}
Let $m,n\in\pinteger$, $q,Q\in\real$, $U\subset\real^n$ an open set, $V\in\varifolds_m(U)$ satisfying Hypothesis \ref{basic hypothesis} and $a\in\support\mass{V}$.
Suppose $m\leq n$, $q>m$ and $Q\geq 1$.

There exist $\varepsilon_0(m,Q)>0$ and $\varepsilon_1(n,m,q,Q)>0$ with the following property:
If $r>0$, $\closedball{a}{r}\subset U$, $\density{m}{\mass{V}}{a}=Q$ and $W\in\varifolds_m(\openball{a}{r})$ satisfies \ref{partition definition}(1) with respect to $V_r\in\varifolds_m(\openball{a}{r})$ are such that
\begin{enumerate}[(i)]
\item $r^{-m}\mass{V}(\closedball{a}{r})\leq (Q+\varepsilon_0)\unitmeasure{m}$;
\item $r^\mq\totalcurvature{V}{\closedball{a}{r}}{q}\leq\varepsilon_1$ and
\item $\support{\mass{W}}\cap\openball{a}{\frac{r}{2}}\neq\emptyset$,
\end{enumerate}
then $a\in\support{\mass{W}}$.

In particular, if $\Pi$ is a partition of $V_r$ then $\cardinality{\Pi'}\leq Q$.
\end{lemma}
\begin{proof}
First we fix $\varepsilon_0,\varepsilon_1>0$ sufficiently small so that $\varepsilon_0\left(1+\frac{1}{2^{m+1}Q}\right)\leq \frac{1}{2^{m+1}}$ and $e^{\Lambda\varepsilon_1}\leq\min\{1+\frac{1}{2^{m+1}Q},\frac{2}{\sqrt{3}}\}$.
Given $0<s<r$ we simplify notation and write 
\begin{equation*}
G_V(s,r)=\int_{(\closedball{a}{r}\setminus\closedball{a}{s})\times\grassmannian{n}{m}}|x-a|^{-m-2}|S^\perp_\sharp(x-a)|^2\integrald V_{(x,S)},
\end{equation*}
where $S_\sharp$ denotes the orthogonal projection associated to a plane $S$.

It follows from Lemma \ref{monotonicity formula} that for all $0<s<r$ we have
\begin{equation*}
e^{\econstant\varepsilon_1\left(\frac{s}{r}\right)^\mq}s^{-m}\mass{V}(\closedball{a}{s}) \leq e^{\econstant\varepsilon_1}r^{-m}\mass{V}(\openball{a}{r})-G_V(s,r).
\end{equation*}
By letting $s$ tend to $0$, using condition $(i)$ and the choices of $\varepsilon_0,\varepsilon_1$ we obtain $G_V(0,r)\leq 2^{-m-2}\unitmeasure{m}$.
Next we observe that $W$ also satisfies Hypothesis \ref{basic hypothesis} by Lemma \ref{inherited hypothesis} and $G_W(0,r)\leq G_V(0,r)$.
Using Lemma \ref{lower bound mass} and \ref{monotonicity formula} on $W$ we have for $0<s<r$,
\begin{equation*}
\begin{aligned}
e^{-\econstant\varepsilon_1 \left(\frac{s}{r}\right)^\mq}s^{-m}\mass{W}(\closedball{a}{s}) & \geq e^{-\econstant\varepsilon_1}r^{-m}\mass{W}(\openball{a}{r})-G_W(s,r)\\
                                                                     & \geq e^{-2\econstant\varepsilon_1}2^{-m}\unitmeasure{m}-2^{-m-2}\unitmeasure{m}\\
																																		 & \geq 3\cdot 2^{-m-2}\unitmeasure{m}-2^{-m-2}\unitmeasure{m}\\
                                                                     & \geq 2^{-m-1}\unitmeasure{m}.\\
\end{aligned}
\end{equation*}
Thus, $\density{m}{\mass{W}}{a}\geq 2^{-m-1}$, which implies $a\in\support{\mass{W}}$.

Finally, it follows from \cite{simon1983:gmtbook}*{17.8} and \ref{basic hypothesis}(i) that $\density{m}{\mass{W}}{a}\geq 1$.
If $\Pi$ is a partition of $V_r$ it follows that 
\begin{equation*}
\density{m}{\mass{V}}{a}\geq \sum_{W\in\Pi'}\density{m}{\mass{W}}{a}\geq \cardinality{\Pi'},
\end{equation*}
which concludes the proof.
\end{proof}

\section{Partition Lemma}

We begin by proving that given a generalised $V$-weakly differentiable function (see \cite{menne2016.1}*{Definition 8.3}) whose image is disconnected it is possible to split $V$ in such a way that the function remains generalised weakly differentiable with respect to the splitting.
We may not guarantee that it will define an indecomposable component (as in \cite{menne2016.1}*{6.2}), but we can still use this to construct a partition of $V$ that preserves generalised weak differentiability with respect to its elements.

\begin{lemma}\label{separation lemma}
Let $U\subset\real^n$ be an open set, $Y$ be a finite dimensional normed vector space, $D\subset Y$ a closed set, $K\subset Y\setminus D$ a compact set, $V\in\varifolds_m(U)$, $\mass{\delta V}$ is a Radon measure, $\mass{\delta V}_{\sing}=0$ and $f\in\vweakfunctions{V}{Y}$ such that
$\mass{V}\left(U\setminus\preimage{f}{D\cup K}\right)=0$ and
$\mass{V}(\preimage{f}{K})>0$.

There exists $W\in\varifolds_m(U)$ a partitioned component of $V$ such that
\begin{enumerate}
\item $\mass{W}(U\setminus \preimage{f}{K})=0$ and
\item $f\in\vweakfunctions{W}{Y}$.
\end{enumerate}
\end{lemma}

\begin{proof}
First we take $\delta>0$ so that $\closedballY{Y}{K}{2\delta}\cap D=\emptyset$ and $\phi:Y\rightarrow \real$ a smooth compactly supported function satisfying:
\begin{enumerate}[(a)]
\item $0\leq \phi\leq 1$;
\item $\phi=1$ on $\closedballY{Y}{K}{\delta}$;
\item $\phi=0$ on $Y\setminus\openballY{Y}{K}{2\delta}$ and
\item $\|\weakd\phi\|\leq 4\delta$.
\end{enumerate}
It follows from \cite{menne2016.1}*{Lemma 8.12} that $\phi\circ f\in\vweakfunctions{V}{\real}$ and $\vweakd{V}{(\phi\circ f)}(x)=\weakd\phi(f(x))\circ\vweakd{V}{f}(x)$ for $\mass{V}$-a.e. $x\in\domain{f}$.
Since $\mass{V}(\preimage{f}{\closedballY{Y}{K}{2\delta}\setminus K})=0$ we further have
\begin{equation*}
\phi\circ f(x)=\left\{
\begin{aligned}
1 &\text{, for }\mass{V}\text{-a.e. }x\in\preimage{f}{\openballY{Y}{K}{2\delta}};\\
0 &\text{, for all }x\in\notpreimage{f}{\openballY{Y}{K}{2\delta}}
\end{aligned}\right.
\end{equation*}
and $\vweakd{V}{(\phi\circ f)}(x)=0$ for $\mass{V}$-almost every $x\in\domain{f}$.

We write $E_s=\{x\in\domain{f}:\phi\circ f(x)>s\}$ and observe that 
\begin{equation*}
\preimage{f}{\closedballY{Y}{K}{\delta}}\subset E_s\subset \preimage{f}{\openballY{Y}{K}{2\delta}} \text{ and } \mass{V}(E_s\setminus \preimage{f}{K})=0
\end{equation*}
 for all $0\leq s< 1$.
Next, from \cite{menne2016.1}*{Lemma 8.29} we compute:
\begin{equation*}
\vboundary{V}{E_s}=\lim_{\varepsilon\rightarrow 0^+}\varepsilon^{-1}\int_{E_s\setminus E_{s+\varepsilon}}\langle\theta(x),\vweakd{V}{(\phi\circ f)}(x)\rangle\integrald\mass{V}_x=0
\end{equation*}
for all $\theta\in\smoothcompact{U,\real^n}$ and $0\leq s< 1$.

Next we define $\Psi:Y\rightarrow Y$ as $\Psi(y)=\phi(y)y$.
As before, it follows from \cite{menne2016.1}*{Lemma 8.12} that $\Psi\circ f\in\vweakfunctions{V}{Y}$,
\begin{equation*}
\Psi\circ f(x)=\left\{
\begin{aligned}
f(x) &\text{, for }\mass{V}\text{-a.e. }x\in\preimage{f}{\openballY{Y}{K}{2\delta}};\\
0 &\text{, for all }x\in\notpreimage{f}{\openballY{Y}{K}{2\delta}}
\end{aligned}\right.
\end{equation*}
and
\begin{equation*}
\vweakd{V}{(\Psi\circ f)}(x)=\left\{
\begin{aligned}
\vweakd{V}{f}(x) &\text{, for }\mass{V}\text{-a.e. }x\in\preimage{f}{\openballY{Y}{K}{2\delta}};\\
0 &\text{, for }\mass{V}\text{-a.e. }x\not\in\preimage{f}{\openballY{Y}{K}{2\delta}}.
\end{aligned}\right.
\end{equation*}

Finally we fix $E=E_0$ and put $W=V\restrict E$.
It follows from the above that $W$ is a partitioned component of $V$ that satisfies condition $(1)$.

Take $\theta\in\smoothcompact{U}{\real^n}$ and $\gamma\in\smoothfunctions{Y}{\real}$ such that $\support{\weakd\gamma}$ is compact, we will prove that $f\in\vweakfunctions{W}{Y}$.

We use $\vboundary{V}{E}=0$, $\mass{\delta V}_{\sing}=0$ and the above to compute
\begin{equation*}
\begin{aligned}
\delta V((\gamma\circ (\Psi\circ f)\theta)) & = (\delta V)\restrict E((\gamma\circ (\Psi\circ f))\theta)+(\delta V)\restrict (U\setminus E)((\gamma\circ (\Psi\circ f))\theta)\\
                                            & = (\delta V)\restrict E((\gamma\circ f)\theta) + \gamma(0)(\delta V)\restrict (U\setminus E)(\theta)\\
                                            & = \delta W((\gamma\circ f))\theta) + \gamma(0)(\delta V)\restrict (U\setminus E)(\theta)
\end{aligned}
\end{equation*}
and
\begin{equation*}
\begin{aligned}
\int_{U\times\grassmannian{n}{m}}\gamma(\Psi\circ f(x))P_{\sharp}\cdot\weakd & \theta(x)\integrald V_{(x,P)}  =\int_{E\times\grassmannian{n}{m}}\gamma(\Psi\circ f(x))P_{\sharp}\cdot\weakd \theta(x)\integrald V_{(x,P)}\\
                                                                              &\qquad\quad +\int_{(U\setminus E)\times\grassmannian{n}{m}}\gamma(\Psi\circ f(x))P_{\sharp}\cdot\weakd \theta(x)\integrald V_{(x,P)}\\
                                                                              &\qquad = \int_{E\times\grassmannian{n}{m}}\gamma(f(x))P_{\sharp}\cdot\weakd \theta(x)\integrald V_{(x,P)}\\
																																							&\qquad\quad +\gamma(0)\int_{(U\setminus E)\times\grassmannian{n}{m}}P_{\sharp}\cdot\weakd \theta(x)\integrald V_{(x,P)}\\
																																							&\qquad = \int\gamma(f(x))P_{\sharp}\cdot\weakd \theta(x)\integrald W_{(x,P)}\\
																																							&\qquad\quad +\gamma(0)(\delta V)\restrict (U\setminus E)(\theta)
\end{aligned}
\end{equation*}

Since $\Psi\circ f\in\vweakfunctions{V}{Y}$ we have
\begin{equation*}
\begin{aligned}
0 & = \delta V((\gamma\circ (\Psi\circ f)\theta)) - \int\gamma(\Psi\circ f(x))P_{\sharp}\cdot\weakd \theta(x)\integrald V_{(x,P)}\\
  &\quad - \int\langle\theta(x),\weakd\gamma(\Psi\circ f(x))\circ\vweakd{V}{(\Psi\circ f)}(x)\rangle\integrald\mass{V}_x\\
	& = \delta W((\gamma\circ f))\theta)-\int\gamma(f(x))P_{\sharp}\cdot\weakd \theta(x)\integrald W_{(x,P)}\\
	&\quad -\int_E\langle\theta(x),\weakd\gamma(f(x))\circ\vweakd{V}{f}(x)\rangle\integrald\mass{V}_x\\
  & = \delta W((\gamma\circ f))\theta)-\int\gamma(f(x))P_{\sharp}\cdot\weakd \theta(x)\integrald W_{(x,P)}\\
	&\quad -\int\langle\theta(x),\weakd\gamma(f(x))\circ\vweakd{V}{f}(x)\rangle\integrald\mass{W}_x,
\end{aligned}
\end{equation*}
which shows that $f\in\vweakfunctions{W}{Y}$ and concludes the proof.
\end{proof}

\begin{lemma}\label{partition lemma}
Let $m,n,Q\in\pinteger$, $r,\lambda,q\in\real$, $U\subset\real^n$ be an open set and $Y$ be a finite dimensional normed vector space.
Suppose $m\leq n$, $q>m$ and $r>0$.
There exist constants $0<\varepsilon_2(\dimension{Y},n,m,Q)<\frac{1}{2}$, and $C_0(\dimension{Y},n,m,q,Q)>0$ with the following property.
If $V\in\varifolds_m(U)$ satisfies Hypothesis \ref{basic hypothesis}, $a\in\support{\mass{V}}$, $\closedball{a}{r}\subset U$, $f\in\vweakfunctions{V}{Y}$ and $0<\lambda\leq\frac{\varepsilon_2}{1+\varepsilon_2}$ satisfy
\begin{enumerate}[(i)]
\item $r^{-m}\mass{V}(\closedball{a}{r})\leq (Q+\frac{1}{4})\unitmeasure{m}$ and
\item $(\mass{V}\restrict\closedball{a}{r})_{(m)}(|\gmeancurvature(V)|)\leq \varepsilon_2$,
\end{enumerate}
then there exist a partition $\Pi$ of $V_{\lambda r}$ satisfying:
\begin{enumerate}
\item $\cardinality\Pi\leq Q$;
\item $f_{\lambda r}\in\vweakfunctions{W}{Y}$ for all $W\in\Pi$ and
\item For all $W\in\Pi$, $\support{(f_{\lambda r})_{\sharp}\mass{W}}$ is pairwise disjoint and
\begin{equation*}
\diam(\support{(f_{\lambda r})_{\sharp}\mass{W}})\leq \left\{
\begin{aligned}
 &C_0\mass{V_r}_{(1)}(\vweakd{V_r}{f_r}) , \text{ if } m=1\\
 &C_0\mass{V_r}_{(q)}(\vweakd{V_r}{f_r})r^{\mq} , \text{ if } m>1
\end{aligned}\right.
\end{equation*}
\end{enumerate}
\end{lemma}

\begin{proof}
Pick $M > \sup\{\dimension{Y},n\}$ sufficiently large so that
$\frac{Q+\frac{1}{4}}{Q+\frac{1}{2}}< 1-M^{-1}$,
let $\Gamma>0$ be given by \cite{menne2016.1}*{Theorem 10.7} and choose $0<\varepsilon_2<\frac{1}{2}$ sufficiently small such that
$\frac{Q+\frac{1}{2}}{Q+1}<(1+\varepsilon_2)^{-m}$ and $\varepsilon_2<\Gamma^{-1}$.

We observe that
\begin{equation*}
\begin{aligned}
\mass{V_r}(\openball{a}{r}) & \leq \left(Q+\frac{1}{4}\right)\unitmeasure{m}r ^m\\
                            & = \left(\frac{Q+\frac{1}{4}}{Q+\frac{1}{2}}\right)\left(\frac{Q+\frac{1}{2}}{Q+1}\right)(Q+1)\unitmeasure{m}r^m\\
                            & < (1-M^{-1})(1+\varepsilon_2)^{-m}(Q+1)\unitmeasure{m}r^m\\
                            & \leq (1-M^{-1})(Q+1)\unitmeasure{m}\left((1-\lambda)r\right)^m
\end{aligned}
\end{equation*}
and $\mass{V_r}(\{x\in\openball{a}{r}:\density{m}{V}{x}<1\})=0$.

We may apply \cite{menne2016.1}*{Theorem 10.7} to $V_r\in\varifolds_m(\openball{a}{r})$ and $A_{(1-\lambda)r}=\{x\in\openball{a}{r}:\openball{x}{(1-\lambda)r}\subset\openball{a}{r}\}=\closedball{a}{\lambda r}$ to obtain a finite set $\Upsilon\subset Y$ such that
\begin{equation*}
\left(\mass{V_r}\restrict\closedball{a}{\lambda r}\right)_{(\infty)}(\distance_{\Upsilon}\circ f_r)\leq \tau_m,
\end{equation*}
where
\begin{equation*}
\tau_m = \left\{
\begin{aligned}
& \Gamma\mass{V_r}_{(1)}\left(\vweakd{V_r}{f_r}\right), \text{ if }m=1\\
& \Gamma^{\frac{m}{\mq}}(1-\lambda)^{\mq}\mass{V_r}_{(q)}\left(\vweakd{V_r}{f_r}\right), \text{ if }m>1.
\end{aligned}
\right.
\end{equation*}
That is, $\mass{V_r}\left(\closedball{a}{\lambda r}\setminus\preimage{f_r}{\closedballY{Y}{\Upsilon}{s}}\right)=0$ for all $s\geq \tau_m$ which further implies $\mass{V_{\lambda r}}\left(\openball{a}{\lambda r}\setminus\preimage{f_{\lambda r}}{\closedballY{Y}{\Upsilon}{s}}\right)=0$ for all $s\geq \tau_m$.

Next we write $\Upsilon=\Upsilon_1\cup\ldots\cup\Upsilon_l$ with $1\leq l\leq Q$ such that $\Upsilon_i\neq\emptyset$, $\closedballY{Y}{\Upsilon_i}{\tau_m}$ is connected and pairwise disjoint.
In particular, we may apply Lemma \ref{separation lemma} for each $D_i=\bigcup_{j\neq i}\closedballY{Y}{\Upsilon_j}{\tau_m}$ and $K_i=\closedballY{Y}{\Upsilon_i}{\tau_m}$ to obtain $W_i\in\varifolds(\openball{a}{\lambda r})$ satisfying \ref{partition definition} with respect to $V_{\lambda r}$ satisfying:
\begin{enumerate}[(a)]
\item $\mass{W_i}(\openball{a}{\lambda r}\setminus\preimage{f_{\lambda r}}{K_i})=0$ and
\item $f_{\lambda r}\in\vweakfunctions{W_i}{Y}$.
\end{enumerate}
In particular, $\support{\mass{W_i}}\subset \preimage{f_{\lambda r}}{K_i}$.

It follows that $\Pi=\{W_i:i=1,\ldots l\}$ is a partition of $V_{\lambda r}$ and
\begin{equation*}
\diam(\support{(f_{\lambda r})_{\sharp}\mass{W_i}}) \leq \diam(K_i) \leq 2Q\tau_m,
\end{equation*}
which concludes the proof with
\begin{equation*}
C_0 = \left\{
\begin{aligned}
& 2Q\Gamma, \text{ if }m=1\\
& 2Q\Gamma^{\frac{m}{\mq}}, \text{ if }m>1.
\end{aligned}
\right.
\end{equation*}

\end{proof}

\begin{theorem}\label{partition theorem}
Let $m,n,Q\in\pinteger$, $r,\lambda,q\in\real$, $U\subset\real^n$ be an open set, $V\in\varifolds_m(U)$ be as in Hypothesis \ref{basic hypothesis}, $a\in\support{\mass{V}}$, $Y$ be a finite dimensional normed vector space and $f\in\vweakfunctions{V}{Y}$.
Suppose $m\leq n$, $q>m$, $r>0$ and let $\varepsilon_2,C_0>0$ be as in Lemma \ref{partition lemma}.
If $\closedball{a}{r}\subset U$, $0<\lambda\leq \frac{\varepsilon_2}{1+\varepsilon_2}$, $r_k=\lambda^k r$ for each $k\in\pinteger$ and $V$ satisfy:
\begin{enumerate}[(i)]
\item $t^{-m}\mass{V}(\closedball{a}{t})\leq (Q+\frac{1}{4})\unitmeasure{m}$ for all $0<t\leq r$ and
\item $\totalcurvature{V}{\closedball{a}{r}}{m}\leq \varepsilon_2$,
\end{enumerate}
then there exist partitions $\Pi_k$ of $V_{r_k}$ for each $k\in\pinteger$ satisfying:
\begin{enumerate}
\item $f_{r_k}\in\vweakfunctions{W}{Y}$ for all $W\in\Pi_k$;
\item For all $W\in\Pi_k$, $\support{(f_{r_k})_\sharp\mass{W}}$ is pairwise disjoint and
\begin{equation*}
\diam(\support{(f_{r_k})_\sharp\mass{W}}) \leq \left\{
\begin{aligned}
 &C_0\mass{V_{r_{k-1}}}_{(1)}(\vweakd{V_{r_{k-1}}}{f_{r_{k-1}}}) , m=1\\
 &C_0\mass{V_{r_{k-1}}}_{(q)}(\vweakd{V_{r_{k-1}}}{f_{r_{k-1}}})r_k^{\mq} , m>1;
\end{aligned}\right.
\end{equation*}
\item For each $Z\in\Pi_{k+1}$ there exists $W\in\Pi_k'$ such that $Z$ is an element of a partition of $W|\parts{\openball{a}{r_{k+1}}\times\grassmannian{n}{m}}$.
\end{enumerate}

\end{theorem}
\begin{proof}
We proceed by induction.
For the sake of notation we put $\Pi_0=\{V_r\}$.
Let $\Pi_1$ be the partition given by Lemma \ref{partition lemma} applied to $V_r$.
Now, suppose $\Pi_k$ is defined for $k>0$ so that it satisfies properties $(1)$-$(3)$.
Observe that $\tilde{\Pi}_k=\{W_{r_{k+1}}:W\in\Pi_k'\text{ and }W_{r_{k+1}}\neq 0\}$ defines a partition of $V_{r_{k+1}}$ and each $W\in\Pi_k'$ satisfy the conditions of Lemma \ref{partition lemma}.
Therefore we can construct $\Pi(W_{r_{k+1}})$ satisfying conditions (1) and (2).
If we let $\Pi_{k+1}=\cup_{W\in\tilde{\Pi}_k}\Pi(W_{r_{k+1}})$ then condition (3) follows trivially.
\end{proof}

\begin{corollary}\label{partition corollary}
Let $m,n,Q\in\pinteger$, $r,q,\sigma,\lambda\in\real$, $U\subset\real^n$ be an open set, $V\in\varifolds_m(U)$ be as in Hypothesis \ref{basic hypothesis}, $a\in\support{\mass{V}}$, $Y$ be a finite dimensional normed vector space and $f\in\vweakfunctions{V}{Y}$.
Suppose $m\leq n$, $q>m$, $\sigma,r>0$ and $\lambda\in(0,1)$.

There exist positive constants $C_1(\dimension{Y},n,m,q,Q)$, $\varepsilon_3(m,Q) $, $\varepsilon_4(\dimension{Y},n,
m,$ $q,Q)<\frac{1}{2}$ and $\Upsilon\subset Y$ with $\cardinality{\Upsilon}\leq Q$ and the following property.
If $\lambda\leq \frac{\varepsilon_4}{1+\varepsilon_4}$, $\closedball{a}{r}\subset U$, $\density{m}{\mass{V}}{a}= Q$ and $V$ satisfy
\begin{enumerate}[(i)]
\item $r^{-m}\mass{V}(\closedball{a}{r})\leq (Q+\varepsilon_3)\unitmeasure{m}$,
\item $\smcurvature{V}{a}{r}{q}\leq\varepsilon_4$ and
\item $r^\mq(\mass{V}\restrict\closedball{a}{r})_{(q)}(|\vweakd{V}{f}|) \leq \sigma$,
\end{enumerate}
then for each $k\in\pinteger$ and $r_k=\lambda^kr$, there exists a partition $\Pi_k\subset\varifolds_m(\openball{a}{r_k})$ of $V_{r_k}$ such that
\begin{enumerate}
\item $\cardinality{\Pi_k'}\leq Q$;
\item $f_{r_k}\in\vweakfunctions{W}{Y}$ for each $W\in\Pi_k$ and
\item For $\mass{V}$-almost every $x\in\openball{a}{r_k}$ we have
\begin{equation*}
\distance^Y(f(x),\Upsilon)\leq C_1\sigma\left(\frac{|x-a|}{\lambda r}\right)^{\mq}.
\end{equation*}
\end{enumerate}
\end{corollary}
\begin{proof}
We begin by letting $\varepsilon_0(m,Q),\varepsilon_1(n,m,q,Q)$ be given by Lemma \ref{support of component contains a} and $\varepsilon_2(\dimension{Y},n,m,Q)$ be given by Lemma \ref{partition lemma}.
Now fix positive constants $\varepsilon_3(m,Q)$ and $\varepsilon_4(\dimension{Y},n,m,q,Q)$ sufficiently small such that $\varepsilon_3<\min\{\frac{1}{8},\varepsilon_0\}$, $e^{\econstant\varepsilon_4}(Q+\frac{1}{8})\leq Q+\frac{1}{4}$, $\varepsilon_4<\varepsilon_1$ and $\varepsilon_4\left((Q+\frac{1}{8})\unitmeasure{m}\right)^{\frac{\mq}{m}}<\varepsilon_2$.

Firstly, it follows from the Monotonicity Formula that for all $0<t\leq r$
\begin{equation*}
\begin{aligned}
t^{-m}\mass{V}(\closedball{a}{t}) &\leq e^{\econstant\varepsilon_4}r^{-m}\mass{V}(\closedball{a}{r})\\
                                  &\leq e^{\econstant\varepsilon_4}(Q+\varepsilon_3)\unitmeasure{m}\\
																	&\leq (Q+\frac{1}{4})\unitmeasure{m}.
\end{aligned}
\end{equation*}
Secondly, it follows from H\"older inequality that
\begin{equation*}
\begin{aligned}
\totalcurvature{V}{\closedball{a}{r}}{m} & \leq \totalcurvature{V}{\closedball{a}{r}}{q}\mass{V}(\closedball{a}{r})^{\frac{\mq}{m}}\\
                                         & \leq \varepsilon_4\left((Q+\varepsilon_3)\unitmeasure{m}\right)^{\frac{\mq}{m}}\\
																				 & \leq \varepsilon_2.
\end{aligned}
\end{equation*}
We then may apply Theorem \ref{partition theorem} to obtain partitions $\Pi_k$ of $V_{r_{k}}$ so that $(2)$ is satisfied and $(1)$ follows from Lemma \ref{support of component contains a}.

Let us first prove $(3)$ for sufficiently large $k$.
For each $Z\in\Pi_{k+1}'$ it follows from \ref{partition theorem}(3) that there exists $W\in\Pi_k'$ such that $Z$ belongs to a partition of $W|\parts{\openball{a}{r_{k+1}}\times\grassmannian{n}{m}}$.
Now observe that Lemma \ref{support of component contains a} implies that for every $W\in\Pi_k'$ there exists at least one $Z\in\Pi_{k+1}'$ such that $\support{W}\cap\support{Z}\neq\emptyset$.
Hence, the above map from $\Pi_{k+1}'$ to $\Pi_{k}'$ is surjective, that is, $\cardinality\Pi_{k+1}'\geq\cardinality\Pi_{k}'$.
Since $\cardinality\Pi_{k}'\leq Q$, then $l=\cardinality\Pi_{k}'$ is eventually constant.
Let $k_0\in\pinteger$ be the first positive integer so that $l=\cardinality\Pi_{k}'$ is constant for all $k\geq k_0$.
For each $k\geq k_0$ we write $\Pi_k'=\{W^k_1,\ldots,W^k_l\}$ ordered so that $W^{k+1}_i$ is an element of a partition of $W^k_i|\parts{\openball{a}{r_{k+1}}\times\grassmannian{n}{m}}$ for each $i=1,\ldots,l$.

We note that when $m=1$, H\"older inequality implies that for all $0<t\leq r$
\begin{equation*}
\begin{aligned}
(\mass{V}\restrict\openball{a}{t})_{(1)}(|\vweakd{V}{f}|) & \leq (\mass{V}\restrict\openball{a}{t})_{(q)}(|\vweakd{V}{f}|)\mass{V}(\closedball{a}{t})^\mq\\
                                                          & \leq \sigma \left((Q+\frac{1}{4})\unitmeasure{m}\right)^\mq r^{-\mq}t^\mq.  
\end{aligned}
\end{equation*}

It follows from \ref{partition theorem}(2) that $\diam(\support{(f_{r_k})_\sharp\mass{W_k^i}}) \leq C_1\sigma  r^{-\mq}r_k^{\mq}$, where $C_1=C_0$ when $m>1$, $C_1=\left((Q+\frac{1}{4})\unitmeasure{m}\right)^\mq C_0$ when $m=1$ and $C_0$ is given by Theorem \ref{partition theorem}.
Since $\support{(f_{r_{k+1}})_\sharp\mass{W^{k+1}_i}}\subset\support{(f_{r_k})_{\sharp}\mass{W^k_i}}$, there exists $y_i\in Y$ such that $\cap_{k=k_0}^\infty\support{(f_{r_k})_\sharp\mass{W^k_i}}=\{y_i\}$.
Hence, for $\mass{W^k_i}$-almost every $x\in\openball{a}{r_k}\setminus\closedball{a}{r_{k+1}}$ we have

\begin{equation*}
\distance^Y(f(x),y_i)\leq C_1\sigma r^{-\mq}r_k^{\mq}\leq C_1\sigma \lambda^{-\mq} r^{-\mq}r_{k+1}^\mq\leq C_1\sigma(\lambda r)^{-\mq}|x-a|^\mq.
\end{equation*}
Since $\mass{W_i^{k+1}}\leq \mass{W_i^{k}}$, the above holds for $\mass{W_i^k}$-almost every $x\in\openball{a}{r_k}$.

Finally, it remains to prove condition (3) for all $k'\leq k_0$.
Let $k'\in\pinteger$ with $k'\leq k_0$ and $W\in\Pi_{k'}$.
It follows from \ref{partition theorem}(3) that there exists $W_{i'}^{k_0}\in\Pi_{k_0}$ such that $\support{\mass{W_{i'}^{k_0}}}\subset\support{\mass{W}}$.
In particular $\support{(f_{r_{k_0}})_\sharp\mass{W_{i'}^{k_0}}}\subset\support{(f_{r_{k'}})_\sharp\mass{W}}$ and $y_{i'}\in\support{(f_{r_{k'}})_\sharp\mass{W}}$
Hence, from \ref{partition theorem}(2) we have
\begin{equation*}
\begin{aligned}
\distance^Y(f(x),\Upsilon) & \leq\distance^Y(f(x),y_{i'})\leq C_1\sigma \left(\frac{r_{k'}}{r}\right)^{\mq}\leq C_1\sigma \left(\frac{r_{k'+1}}{\lambda r}\right)^\mq\\
                           & \leq C_1\sigma \left(\frac{|x-a|}{\lambda r}\right)^\mq
\end{aligned}
\end{equation*}
for $\mass{W}$-almost every $x\in\openball{a}{r_{k'}}\setminus\closedball{a}{r_{k'+1}}$, which concludes the proof.
\end{proof}

\section{Lipschitz Approximation}

\noindent\textbf{Notation.} 
Henceforth we fix $Y=\{A\in\homom{\real^n}{\real^n}:A^*=A\}$ and for $P\in\grassmannian{n}{m}$ we write $P_{\sharp}\in Y$ its corresponding orthogonal projection.
Given $a\in\real^n$ and $s,r>0$ we write $\planeopenball{P}{a}{r}=\{x\in\real^n:|P_{\sharp}(x-a)|\leq r\}$ and similarly $\planeclosedball{P}{a}{r}$.
We also define the truncated cylinder with base $P$ centered at $a$ with radius $r$ and height $s$ as
\begin{equation*}
\tcylinder{P}{a}{r}{s}=\{x\in\real^n:|P_\sharp(x-a)|< r \text{ and } |P^\perp_\sharp(x-a)|< s\}.
\end{equation*}

\begin{definition}
Let $n,m\in\pinteger$ with $m\leq n$ and $U\subset \real^n$ an open set.
Given $V\in\ivarifolds_m(U)$ we write $T_V=\{x\in U:\tangentmap{m}{\mass{V}}{x}\in\grassmannian{n}{m}\}$ and $\tmap{V}:T_V\rightarrow Y$ with $\tmap{V}(x)=\tangentmap{m}{\mass{V}}{x}_\sharp$.

We say that $V$ is a curvature varifold if $\mass{\delta V}$ is a Radon measure absolutely continuous with respect to $\mass{V}$ and $\tmap{V}\in\vweakfunctions{V}{Y}$, that is, $\tmap{V}$ is generalised $V$ weakly differentiable.
Given $x\in T_V$, we write $\gsff{V}(x)=\vweakd{V}{\tmap{V}}(x)$ and call it the generalised second fundamental form of $V$ at $x$. 
\end{definition}

\begin{remark}
It follows from \cite{menne2016.1}*{Theorem 15.6} that this definition is equivalent to the original definition by Hutchinson \cite{hutchinson1986}*{Definition 5.2.1}.
Furthermore,  $\gmeancurvature(V,x)=\trace{\tmap{V}(x)}(\gsff{V}(x))$ for $\mass{V}$-almost every $x\in U$.
\end{remark}

We begin by discussing the structure theorem for curvature varifolds with null second fundamental form.
The original statement \cite{hutchinson1986.2}*{Corollary to Theorem 3.4} is incorrect in two directions.
Firstly, the global statement is not true as one may consider $U=\openball{0}{3}\setminus\closedball{0}{1}$ in $\real^2$ and $V$ is the varifold such that $\mass{V}=\hausdorff^1\restrict(1,3)\times\{0\}$ so that $V$ is not given by a line restricted to $U$.
Secondly, the finiteness of the number of planes is also false without further assumptions.
As a counter example we consider $U=\openball{0}{1}$ in $\real^2$ and $V$ given by a sequence of lines accumulating at $\{1\}\times\real$.
We can even choose such lines so that the mass of $V$ is finite.
Nevertheless, the local statement remains true and in fact it can be proved independently of the above results.

\begin{figure}[H]
\centering
\begin{subfigure}{.5\textwidth}
  \centering
  \begin{tikzpicture}[scale=0.8]
\draw[dashed]  (0,0) node (v1) {} circle (3);
\draw[blue,very thick] (1,0) -- (3,0);
\draw[dashed]  (v1) circle (1);

\node [label={[label distance = 0.1 cm]90:$V$}] at (2,0) {};

\end{tikzpicture}
  \caption*{Null second fundamental form but\\ not a plane restricted to $U$.}
  \label{fig2:sub1}
\end{subfigure}%
\begin{subfigure}{.5\textwidth}
  \centering
  \begin{tikzpicture}[scale=0.8]
\draw[dashed]  (0,0) node (v1) {} circle (3);
\draw[blue,very thick] (0,-3) -- (0,3);

\clip (0,0) circle (3);
\pgfmathsetmacro{\i}{0}
\pgfmathsetmacro{\j}{0}
\foreach \j in {1,...,5} 
    \draw[blue,very thick] (3-2/\j,-3) -- (3-2/\j,3);
    
\draw[dotted, blue, thick] (2.5,0) -- (3,0);

\node [label={[label distance = 0.1 cm]180:$V$}] at (0,0) {};

\end{tikzpicture}
  \caption*{Finite density and mass in $U$ but not given by finitely many planes.}
  \label{fig2:sub2}
\end{subfigure}
\label{fig2:test}
\end{figure}

These examples illustrate that one should consider connected components of a decomposition of $V$ and density assumptions in order to write up an appropriate global statement.

\begin{theorem}\label{structure theorem}
Let $n,m\in\pinteger$ with $m\leq n$, $U\subset \real^n$ an open set and $V\in\ivarifolds_m(U)$ be a curvature varifold.
If $\gsff{V}(x)=0$ for $\mass{V}$-almost every $x\in U$, then for each $a\in U$ there exists $r>0$, $N\in\pinteger$, $\{k_1,\ldots,k_N\}\subset\pinteger$ and $A_1,\ldots,A_N$ affine planes such that
\begin{equation*}
V\restrict\openball{a}{r}=\sum_{j=1}^N\setvarifold{A_j}{k_j}\restrict \openball{a}{r}.
\end{equation*}
\end{theorem}
\begin{proof}
It follows from \cite{menne2016.1}*{Theorem 8.34} that there exists a decomposition $\Theta$ of $V$ in the sense of \cite{menne2016.1}*{Definition 6.9} (which is the same as a partition except that each element is indecomposable) and $P:\Theta\rightarrow\grassmannian{n}{m}$ such that for each $W\in\Theta$ the tangent map satisfies $\tmap{W}(x)=P(W)_\sharp$ for $\mass{W}$-almost every $x\in U$.
Next we observe that $\delta W=0$ so it follows from \cite{menne2010}*{3.1 p.379} that for each $a\in\support{\mass{W}}$ we can find $r>0$ such that $W\restrict\openball{a}{r}=\setvarifold{P(W)+a}{\density{m}{W}{a}}\restrict \openball{a}{r}$ and $\density{m}{W}{a}\in\pinteger$, which concludes the proof.
\end{proof}

\begin{remark}
One could obtain a similar proof using our partition Theorem \ref{partition theorem} and the height estimates given by Lemma \ref{conical estimate} below.
The above proof illustrates that the structure theorem is independent of Hutchinson's results in \cite{hutchinson1986.2}.
Although it relies on other non-trivial results such as \cite{menne2016.1}*{Theorem 8.34} and the equivalency of the definitions of curvature varifolds \cite{menne2016.1}*{Theorem 15.6}.
\end{remark}

We continue by proving a height estimate and technical lemmas to derive estimates on truncated cylinder from corresponding estimates on closed balls in order to obtain the graphical Lipschitz approximation.

\begin{lemma}\label{conical estimate}
Let $m,n,Q\in\pinteger$, $\delta_1,\delta_2,\sigma \in\real$, $U\subset\real^n$ be an open set.
Suppose $m\leq n$, $r>0$ and $\delta_1,\delta_2,\sigma\in(0,1)$.

There exist positive constants $\lambda_0(m,\delta_2)$ and $\varepsilon_5(n,m,Q,\delta_1,\delta_2,\sigma)$ with the following property.
If $\closedball{a}{r}\subset U$, $V\in\ivarifolds_m(U)$ satisfies Hypothesis \ref{basic hypothesis}, $a\in\support{\mass{V}}$ and $P\in\grassmannian{n}{m}$ satisfy:
\begin{enumerate}[(i)]
\item $(Q-\delta_1)\unitmeasure{m}\leq t^{-m}\mass{V}\closedball{a}{t} \leq (Q+\delta_2)\unitmeasure{m}$;
\item $\totalcurvature{V}{\closedball{a}{t}}{m}\leq\varepsilon_5$ and
\item $\int_{\closedball{a}{t}\times\grassmannian{n}{m}}|S-P|\integrald V_{(x,S)}\leq \varepsilon_5\mass{V}(\closedball{a}{t})$,
\end{enumerate}
for all $0<t\leq r$.

Then $\mass{V}(\openball{a}{\lambda_0 r}\cap\{x\in \real^n:|P^\perp_{\sharp}(x-a)|>\sigma|P_\sharp(x-a)|\})=0$.
\end{lemma}
\begin{proof}
Fix $0<\lambda_0<1$ such that $(1-\lambda_0^2)^{\frac{m}{2}}\geq \frac{1+\delta_2}{2}$ and suppose by contradiction that such $\varepsilon_5>0$ does not exist.
Let $\{\epsilon_i\}_{i\in\pinteger}$ be an arbitrary sequence of positive number such that $\epsilon_i\rightarrow 0$.
There exist sequences $V_i\in\ivarifolds_m(U)$, $a_i\in\support{\mass{V_i}}$, $P_i\in\grassmannian{n}{m}$ and $r_i>0$ satisfying $(i),(ii)$ and $(iii)$ with respect to $\epsilon_i$ and
\begin{equation*}
\mass{V_i}(\openball{a_i}{\lambda_0 r_i}\cap\{x\in \real^n:|{P_i}^\perp_{\sharp}(x-a_i)|>\sigma|{P_i}_\sharp(x-a_i)|\})>0.
\end{equation*}
Since these are all scale invariant, we may assume by translation, rotation and dilation that $a_i=0$, $P_i=P$ and $r_i=1$ are fixed for all $i\in\pinteger$.

Take $x_i\in\support{\mass{V_i}}\cap\openball{0}{\lambda_0}\cap\{x\in \real^n:|P^\perp_{\sharp}(x)|>\sigma|P_\sharp(x)|\}$.
By taking a subsequence we may assume that $x_i\rightarrow \bar{x}\in\closedball{0}{\lambda_0}\cap\{x\in \real^n:|P^\perp_{\sharp}(x)|\geq\sigma|P_\sharp(x)|\}$.

First we assume $\bar{x}\neq 0$.
It follows from \cite{allard1972}*{Theorem 6.4} that $V_i\restrict\closedball{0}{1}$ converges to a stationary varifold $V\in\ivarifolds_m(\closedball{0}{1})$ such that $S=P$ for $V$-almost every $(x,S)\in\closedball{0}{1}\times\grassmannian{n}{m}$ and $0,\bar{x}\in\support{\mass{V}}$.
Therefore, \cite{almgren2000}*{Theorem 3.6} implies that there exist $\alpha_1,\ldots,\alpha_k\in\pinteger$ and $v_1,\ldots,v_k\in P^\perp$ such that
\begin{equation*}
V=\sum_{j=1}^k\setvarifold{(P+v_j)\cap\closedball{0}{1}}{\alpha_j}.
\end{equation*}
It follows from $(i)$ that $Q-\delta_1\leq\density{m}{V}{0}\leq Q+\delta_2$ so we may further assume that $v_1=0$ and $\alpha_1=Q$.
By assumption $\bar{x}\neq 0$ so we may further assume that $v_2=P^\perp\sharp \bar x$ and $0<|v_2|\leq\lambda_0$.
In particular we have
\begin{equation*}
\begin{aligned}
(Q+\delta_2)\unitmeasure{m}& \geq\mass{V}(\closedball{0}{1})\\
                           & \geq Q\unitmeasure{m}+ (1-|v_2|^2)^{\frac{m}{2}}\unitmeasure{m}\\
                           & \geq Q\unitmeasure{m}+ (1-\lambda_0^2)^{\frac{m}{2}}\unitmeasure{m},
\end{aligned}
\end{equation*}
which contradicts our choice of $\lambda_0$.

Now, assume $\bar{x}=0$.
Let $d_i=|x_i|$, by further dilation we may assume $r_i=d_i^{-1}$, $\tilde{x_i}=d_i^{-1}x_i$ and $\tilde{x_i}\rightarrow \tilde{x}$ with $\tilde{x}=1$.
Fix $\tau_0>0$ sufficiently large such that $\frac{((1+\tau_0)^2-1)^{\frac{m}{2}}}{(1+\tau_0)^m}=\frac{1+\delta_2}{2}$.
For $i$ suffiently large we have $d_i^{-1}>1+\tau_0$ so we may proceed as before and conclude that $V_i\restrict\closedball{0}{1+\tau_0}$ converges to a stationary varifold $\tilde{V}\in\ivarifolds_m(\closedball{0}{1+\tau_0})$ such that $S=P$ for $\tilde{V}$-almost every $(x,S)\in\closedball{0}{1}\times\grassmannian{n}{m}$ and $0,\tilde{x}\in\support{\mass{\tilde{V}}}$ and also apply \cite{almgren2000}*{Theorem 3.6} to $\tilde{V}$.

As above, we may further assume that $v_1=0$ and $\alpha_1=Q$.
By assumption $\tilde{x}\neq 0$ so we may further assume that $v_2=P^\perp\sharp \bar x$ and $0<|v_2|\leq 1$.
In particular we have
\begin{equation*}
\begin{aligned}
(Q+\delta_2)\unitmeasure{m}(1+\tau_0)^m& \geq\mass{V}(\closedball{0}{1+\tau_0})\\
                           & \geq Q\unitmeasure{m}(1+\tau_0)^m+ ((1+\tau_0)^2-|v_2|^2)^{\frac{m}{2}}\unitmeasure{m}\\
                           & \geq Q\unitmeasure{m}(1+\tau_0)^m+ ((1+\tau_0)^2-1)^{\frac{m}{2}}\unitmeasure{m},
\end{aligned}
\end{equation*}
which contradicts our choice of $\tau_0$ and concludes the proof.
\end{proof}

\begin{lemma}\label{cylinder estimate}
Let $m,n,Q\in\pinteger$, $r,\delta_1,\delta_2,\delta_3\in\real$, $U\subset\real^n$ be an open set, $P\in\grassmannian{n}{m}$, $V\in\ivarifolds_m(U)$ satisfying Hypothesis \ref{basic hypothesis} and $a\in\support{\mass{V}}$.
Suppose $m\leq n$, $r>0$, $\delta_1,\delta_2\in(0,1)$ and $\delta_3\in(0,\frac{1}{4})$.

There exist positive constants $\lambda_1(m,Q,\delta_2,\delta_3)$ and $\varepsilon_6(n,m,Q,\delta_1,\delta_2,\delta_3)$ with the following property.
If $\closedball{a}{r}\subset U$ and $V$ satisfy:
\begin{enumerate}[(i)]
\item $(Q-\delta_1)\unitmeasure{m}\leq t^{-m}\mass{V}\closedball{a}{t} \leq (Q+\delta_2)\unitmeasure{m}$;
\item $\totalcurvature{V}{\closedball{a}{t}}{m}\leq\varepsilon_6$ and
\item $\int_{\closedball{a}{t}\times\grassmannian{n}{m}}|S-P|\integrald V_{(x,S)}\leq \varepsilon_6\mass{V}(\closedball{a}{t})$,
\end{enumerate}
for all $0<t\leq r$.

Then, for $s=\frac{\lambda_1 r}{2}$ we have
\begin{enumerate}[(1)]
\item $(Q-\delta_1)\unitmeasure{m}\leq s^{-m}\mass{V}(\tcylinder{P}{a}{s}{s})\leq(Q+\frac{1+\delta_2}{2})\unitmeasure{m}$;
\item $\mass{V}(\tcylinder{P}{a}{s}{(1+\delta_3)s}\setminus\tcylinder{P}{a}{s}{(1-2\delta_3)s})=0$ and
\item $s^{-m}\mass{V}(\closedball{\tcylinder{P}{a}{s}{s}}{2s})\leq 4^m(Q+\delta_2)\unitmeasure{m}$.
\end{enumerate}
\end{lemma}
\begin{proof}
Choose $0<\sigma_0<1-2\delta_3$ sufficiently small so that $(1+\sigma_0^2)^{\frac{m}{2}}<\frac{Q+\frac{1+\delta_2}{2}}{Q+\delta_2}$ and $\varepsilon_6=\varepsilon_5(n,m,Q,\delta_1,\delta_2,\sigma_0)$ be given by Lemma \ref{conical estimate}.
Let $\lambda_0(m,\delta_2)$ be given by Lemma \ref{conical estimate} and fix $\lambda_1=\min\{\lambda_0,\frac{2}{2+\sqrt{2}},\frac{2}{(1+\sigma_0^2)^{\frac{1}{2}}}\}$.
For simplicity we write $\tcylshort{s}{s}=\tcylinder{P}{a}{s}{s}$ and $\cbshort{s}=\closedball{a}{s}$.

First note that $\cbshort{s}\subset\tcylshort{s}{s}$ and Lemma \ref{conical estimate} implies 
\begin{equation*}
\support{\mass{V}}\cap\tcylshort{s}{s}\subset\tcylshort{s}{\sigma_0 s}\subset\cbshort{(1+\sigma_0^2)^{\frac{1}{2}}s},
\end{equation*}
which proves $(1)$.

Secondly, we note that $\tcylshort{s}{(1+\delta_3)s}\subset\cbshort{\lambda_0 r}$ and $\support{\mass{V}}\cap\tcylshort{s}{s}\subset\tcylshort{s}{(1-2\delta_3)s}$, hence $\support{\mass{V}}\cap(\tcylshort{s}{(1+\delta_3)s}\setminus\tcylshort{s}{(1-2\delta_3)s})=\emptyset$, which proves $(2)$.

Finally, it follows directly that $\closedball{\tcylinder{P}{a}{s}{s}}{2s})\subset\cbshort{(2+\sqrt{2})s}\subset\cbshort{r}$ so that $(3)$ is proved.
\end{proof}

\begin{definition}
Let $n,m,Q\in\pinteger$, $U\subset\real^m$ be an open set and $u:U\rightarrow\qspace{Q}{\real^{m-n}}$ be a $Q$-valued function.
We write 
$\graph{u}=\{(x,y)\in U\times\real^{n-m}:y\in\support{u(x)}\}$,
$\theta_u:U\times\real^{n-m}\rightarrow\real$ as 
$\theta_u(x,y)=\sum_{y\in\support{u(x)}}\density{0}{|u(x)|}{y}$ and
define the associated varifold to $u$ as $\graphvarifold{u}=\setvarifold{\graph{u}}{\theta_u}$.
\end{definition}

The following is a reduced version of \cite{menne2013.1}*{Lemma 4.1} adapted to our specific set-up.

\begin{lemma}[\cite{menne2013.1}*{Lemma 4.1}]\label{lipschitz lemma}
Let $n,m,Q\in\pinteger$, $r,\delta,L\in\real$, $U\subset\real^n$ be an open set, $P\in\grassmannian{n}{m}$, $V\in\ivarifolds_m(U)$ satisfying Hypothesis \ref{basic hypothesis} and $a\in\support{\mass{V}}$.
Suppose $m\leq n$, $r,L>0$ and $\delta\in(0,1)$.

There exist positive constants $\lambda_2(m,Q,\delta)\in(0,\frac{1}{2})$ and $\varepsilon_7(n,m,Q,\delta,L)$ with the following property.
If $\closedball{a}{r}\subset U$ and $V$ satisfy:
\begin{enumerate}[(i)]
\item $(Q-\delta)\unitmeasure{m}\leq t^{-m}\mass{V}\closedball{a}{t} \leq (Q+\delta)\unitmeasure{m}$;
\item $\totalcurvature{V}{\closedball{a}{t}}{m}\leq\varepsilon_7$ and
\item $|\tmap{V}(x)-P|\leq\varepsilon_7$ for $\|V\|$-almost every $x\in\closedball{a}{r}$,
\end{enumerate}
for all $0<t\leq r$.

Then, for $s=\frac{\lambda_2 r}{2}$ there exist $Z\subset\planeopenball{P}{a}{s}$ a Lipschitz function $u:Z\rightarrow\qspace{Q}{\real^{n-m}}$ such that
\begin{enumerate}[(1)]
\item $Z$ is $\lebesgue{m}$-measurable and $\lebesgue{m}(\planeopenball{P}{a}{s}\setminus Z)=0$;
\item $\lip{u}\leq L$;
\item $V\restrict\tcylinder{P}{a}{s}{s}=\graphvarifold{u}$;
\item $u$ is approximately strongly affinely approximable at $\lebesgue{m}$-almost every $z\in Z$ and
\item For $\lebesgue{m}$-almost every $z\in Z$ we have $\tmap{V}(x)=\textup{Tan}(\graph{\affineap{u}(z)},x)$ with $x\in\graph{u}$ and $z=a+P_\sharp(x-a)$.
\end{enumerate}
\end{lemma}
\begin{proof}
Let $\lambda_1(m,Q,1-\delta,\frac{1-\delta}{2})$ and $\varepsilon_6(n,m,1-\delta,\frac{1-\delta}{2},\frac{1}{8})$ be given by Lemma \ref{cylinder estimate} and $\bar\varepsilon$ given by \cite{menne2013.1}*{Lemma 4.1} with respect to $\delta_1=1-\delta$, $\delta_2=\frac{1-\delta}{2}$, $\delta_3=0$, $\delta_4=\frac{1}{8}$, $M=4^m(Q+\delta)$ and $S=P$.

If we fix $\lambda_2=\lambda_1$ and $\varepsilon_7=\min\{\varepsilon_6,\frac{\bar\varepsilon}{2}\}$ it follows from Lemma \ref{cylinder estimate} that the conditions of \cite{menne2013.1}*{Lemma 4.1} are satisfied with $s=\frac{\lambda_2 r}{2}$.
Therefore, we obtain a $\lebesgue{m}$-measurable set $Z\subset\planeopenball{P}{a}{s}$ and a Lipschitz function $u:Z\rightarrow\qspace{Q}{\real^{n-m}}$ satisfying $(2),(3)$ and $(5)$ (see \cite{menne2013.1}*{Lemma 4.1(4),(5)}).
Furthermore, it follows from (ii) and (iii) that the exceptional set $B$ in \cite{menne2013.1}*{Lemma 4.1} is in fact empty, so that \cite{menne2013.1}*{Lemma 4.1(6)} implies $(1)$ and $(3)$.
\end{proof}

Next we prove the uniqueness of the tangent cone for curvature varifolds with sufficiently integrable second fundamental form.

\begin{lemma}\label{unique tangent cone}
Let $n,m\in\pinteger$, $r,q, C\in\real$, $U\subset\real^n$ be an open set, $P\in\grassmannian{n}{m}$, $V\in\ivarifolds_m(U)$ satisfying Hypothesis \ref{basic hypothesis} and $a\in\support{\mass{V}}$.
Suppose $m\leq n$, $q>m$ and $r,C>0$.
If $\closedball{a}{r}\subset U$ and $V$ satisfies
\begin{enumerate}[(i)]
\item $r^{-m}\mass{V}(\closedball{a}{r})\leq C$;
\item $r^\mq\totalcurvature{V}{\closedball{a}{r}}{q}\leq C$ and
\item $|\tmap{V}(x)-P|\leq C|x-a|^\mq$ for $\mass{V}$-almost every $x\in\closedball{a}{r}$.
\end{enumerate}
Then, there exists $Q\in\pinteger$ such that $\setvarifold{P}{Q}$ is the unique tangent tangent cone of $V$ at $a$.
\end{lemma}
\begin{proof}
First, note that Lemma \ref{monotonicity formula} implies that $t^{-m}\mass{V}(\closedball{a}{t})\leq C'$ for all $0<t\leq r$ and some $C'>0$.
Hence
\begin{equation*}
\begin{aligned}
\mass{\delta V}(\closedball{a}{t}) & \leq \int_{\closedball{a}{t}}|\gmeancurvature(V,x)|\integrald\mass{V}_x\\
                                   & \leq \totalcurvature{V}{\closedball{a}{t}}{q}\mass{V}(\closedball{a}{t})^{1-\frac{1}{q}}\\
                                   & \leq t^{m-1}C'\left(\frac{t}{r}\right)^\mq,
\end{aligned}
\end{equation*}
for all $0<t\leq r$ and some $C'>0$.
In particular, $\density{m-1}{\mass{\delta V}}{a}=0$.

Let $C\in\varifolds_m(\real^n)$ be an arbitrary tangent cone of $V$ at $a$, in particular property $(ii)$ implies that $C$ must also be stationary.
Since $V$ is an integral varifold then so is $C$ by \cite{allard1972}*{Theorem 6.4}.
It further follows from \cite{allard1972}*{\textsection 6.5, Theorem 5.2(2)} that $x\in S$ for $C$-almost every $(x,S)\in\real^n\times\grassmannian{n}{m}$ and property $(iii)$ implies that $\tmap{C}(x)=P$ for $\mass{C}$-almost every $x\in\real^n$ so that $\support{\mass{C}}\subset P$.

The conclusion follows from the Constancy Theorem \cite{allard1972}*{Theorem 4.6(3)}.
\end{proof}

\begin{remark}
The uniqueness of tangent map in the form of \cite{hutchinson1986.2}*{Theorem 3.4} follows directly from the above Lemma \ref{unique tangent cone} and Corollary \ref{partition corollary} as it can be seen in the proof of the next theorem.
\end{remark}

Finally we conclude this section with the proof of the Lipschitz approximation for curvature $m$-varifolds with $L^q$-integrable second fundamental form and $q>m$.

\begin{theorem}
Let $n,m,Q\in\pinteger$, $r,q,L\in\real$, $U\subset\real^n$ be an open set, $V\in\ivarifolds_m(U)$ is a curvature varifold satisfying Hypothesis \ref{basic hypothesis} and $a\in\support{\mass{V}}$.
Suppose $m\leq n$, $q>m$, $r,L>0$ and $\density{m}{\mass{V}}{a}=Q$.

There exist positive constants $\lambda_3(n,m,q,Q)$ and $\varepsilon_8(n,m,q,Q,L)$ with the following property.
If $\closedball{a}{r}\subset U$ and $V$ satisfies:
\begin{enumerate}[(i)]
\item $r^{-m}\mass{V}(\closedball{a}{r})\leq (Q+\varepsilon_8)\unitmeasure{m}$ and
\item $r^\mq(\mass{V}\restrict{\closedball{a}{r}})_{(q)}(\gsff{V})\leq \varepsilon_8$.
\end{enumerate}
Then, for $s=\frac{\lambda_3 r}{2}$ there exist a finite set $\Upsilon\subset\grassmannian{n}{m}$ with $\cardinality{\Upsilon}\leq Q$ and for each $P\in\Upsilon$ there exist $Z_P\subset\planeopenball{P}{a}{s}$, $Q_P\in\pinteger$ with $\sum Q_P=Q$ and $u_P:Z_P\rightarrow\qspace{Q_P}{\real^{n-m}}$ Lipschitz function for each $P\in\Upsilon$ such that
\begin{enumerate}[(1)]
\item $Z_P$ is $\lebesgue{m}$-measurable and $\lebesgue{m}(\planeopenball{P}{a}{s}\setminus Z_P)=0$;
\item $\lip{u_P}\leq L$;
\item $V\restrict\openball{a}{s}=\graphvarifold{u_P}\restrict\openball{a}{s}$;
\item $u_P$ is approximately strongly affinely approximable at $\lebesgue{m}$-almost every $z\in Z_P$ and
\item For $\lebesgue{m}$-almost every $z\in Z_P$ we have $\tmap{V}(x)=\textup{Tan}(\graph{\affineap{u_P}(z)},x)$ with $x\in\graph{u_P}$ and $z=a+P_\sharp(x-a)$.
\end{enumerate}
\end{theorem}
\begin{proof}
Let $\varepsilon_3(m,Q)$, $\varepsilon_4(\frac{n(n+1)}{2},n,m,q,Q)$, $C_1(\frac{n(n+1)}{2},n,m,q,Q)$ be given by Corollary \ref{partition corollary} and $\varepsilon_7(n,m,Q,\frac{1}{2},L)$, $\lambda_2(m,Q,\frac{1}{2})$ be given by Lemma \ref{lipschitz lemma}.
Now choose $\lambda_3'\leq\min\{\lambda_2,\frac{\varepsilon_4}{1+\varepsilon_4}\}$ and 
$\varepsilon_8\leq\min\{\frac{\varepsilon_3}{2},\frac{\varepsilon_4}{m},\frac{1}{4}\}$
 sufficiently small so that $\frac{C_1\varepsilon_8}{{\lambda_3'}^\mq}\leq\varepsilon_7$, 
$e^{\econstant\varepsilon_8}\leq\min\{\frac{Q+\varepsilon_3}{Q+\varepsilon_8},\frac{Q+\frac{3}{2}\varepsilon_8}{Q+\varepsilon_8}\frac{8Q}{8Q-1},\frac{Q}{Q-\frac{3}{8}}\}$
 and $(\frac{\varepsilon_8}{m})^m(Q+\frac{1}{2})^\mq\unitmeasure{m}^\mq\leq\varepsilon_7$.

First we note that the Monotonicity Formula \ref{monotonicity formula} implies that
\begin{equation*}
(Q-\frac{3}{8})\unitmeasure{m}\leq t^{-m}\mass{V}(\closedball{a}{t})\leq (Q+\frac{3}{8})\unitmeasure{m}
\end{equation*}
for all $0<t\leq r$.

We write $r'=\lambda_3' r$ and let $\Pi_1$ be the partition of $V_{r'}$ given by Corollary \ref{partition corollary}.
For each $W\in \Pi_1$ it follows from $\ref{partition corollary}(3)$ that we may find $P_W\in\Upsilon$ so that
\begin{equation*}
|\tmap{W}(x)-P_W|\leq C_1\varepsilon_8\left(\frac{|x-a|}{\lambda_3' r}\right)^{\mq}
\end{equation*}
for $\mass{W}$-almost every $x\in\openball{a}{\lambda_3' r}$.

If $W\in\Pi_1'$, it follows from Lemma \ref{support of component contains a} that $a\in\support{\mass{W}}$.
Therefore, from Lemma \ref{unique tangent cone} we have that $P_W$ is the unique tangent cone of $W$ at $a$, $Q_W=\density{m}{\mass{W}}{a}\in\pinteger$ and $\sum_{W\in\Pi_1'}Q_W=Q$.

Second, we claim that ${r'}^{-m}\mass{W}(\closedball{a}{r'})\leq (Q_W+\frac{1}{2})\unitmeasure{m}$.
In fact, we define
\begin{equation*}
\overline{\Pi_1}=\{W\in\Pi_1':{r'}^{-m}\mass{W}(\closedball{a}{r'})\leq (Q_W+\frac{1}{2})\unitmeasure{m}\}.
\end{equation*}
It follows from the Monotonicity Formula and the choice of $\varepsilon_8$ that $\mass{\overline{W}}(\closedball{a}{r'})\geq (1-\frac{1}{8Q})Q_{\overline{W}}\unitmeasure{m}$ for each $\overline{W}\in\overline{\Pi_1}$.
Suppose by contradiction that $\Pi_1'\setminus\overline{\Pi_1}\neq\emptyset$ and put $\overline{V}=V_{r'}-\sum_{\overline{W}\in\overline{\Pi_1}}\overline{W}$.
On one hand we have that 
\begin{equation*}
{r'}^{-m}\mass{\overline{V}}(\closedball{a}{r'})\leq (Q+\frac{3}{8}-(1-\frac{1}{8Q})\sum_{\overline{W}\in\overline{\Pi_1}}Q_{\overline{W}})\unitmeasure{m}.
\end{equation*}
On the other hand we have
\begin{equation*}
{r'}^{-m}\mass{\overline{V}}(\closedball{a}{r'}) > \sum_{W\in\Pi_1'\setminus\overline{\Pi_1}}(Q_W+\frac{1}{2})\unitmeasure{m}=(Q+\frac{\cardinality{\Pi_1'\setminus\overline{\Pi_1}}}{2}-\sum_{\overline{W}\in\overline{\Pi_1}}Q_{\overline{W}})\unitmeasure{m}.
\end{equation*}
This proves the claim.

Once again, the Monotonicity Formula implies that
\begin{equation*}
t^{-m}\mass{W}(\closedball{a}{t})\geq \frac{(8Q-1)Q_W\unitmeasure{m}}{8Q}\geq (Q_W-\frac{1}{8})\unitmeasure{m}
\end{equation*}
for all $0<t\leq r'$.

Finally, it follows from $(ii)$ and the choice of $\varepsilon_8$, together with Holder inequality, that for each $W\in\Pi_1'$ the conditions of Lemma \ref{lipschitz lemma} are satisfied with respect to $P_W$ and $s'=\frac{\lambda_2 r'}{2}$.
We obtain $Z_W\subset\planeopenball{P_W}{a}{s'}$ and $u_W':Z_W\rightarrow\qspace{Q_W}{\real^{n-m}}$ satisfying \ref{lipschitz lemma}(1)-(5).
We conclude the proof by choosing $\lambda_3=\lambda_2\lambda_3'$.
\end{proof}

\begin{remark}
We observe that the above also proves that we may find $\delta$ sufficently small depending on the tangent cone of $V$ at $a$ so that each $\graphvarifold{u_P}$ is the graph of a function defined over a single plane $P_0$.
In which case $V\restrict\tcylinder{P_0}{a}{\delta}{\delta}$ can be written as the graph of a single $Q$-valued function defined on $P_0$
\end{remark}

\appendix

\section{Counter-example}
The following examples addresses the mistake in the arguments of \cite{hutchinson1986}*{Lemma 3.3}.
More specifically at the end of the proof when it is argued that there must exist at most finitely many planes at which the tangent map of the varifold must accumulate.

\begin{lemma}
Let $\varepsilon>0$ and $f,g:(0,\infty)\rightarrow\real$ be positive functions such that $g$ is non-decreasing and $\lim_{t{\rightarrow 0^+}}f(t)=0$.
There exists a family of sets $S(\rho)\subset\real$ for all $\rho\in(0,\infty)$ satisfying:
\begin{enumerate}[(1)]
\item $S(\tau)\subset S(\rho)\subset \real$ for all $0<\tau\leq\rho$;
\item For all $\rho\in(0,\infty)$ there exist $P,Q\in\real$ such that
\begin{equation*}
S(\rho)\subset\mathbf{B}(P,g(\rho))\cup\mathbf{B}(Q,g(\rho));
\end{equation*}
\item If $A\subset\real$ is such that for all $0<\rho\leq\varepsilon]$ and $P\in S(\rho)$ there exists $Q\in A$ with $|Q-P|\leq f(\rho)$, then $A$ is not a finite set.
\end{enumerate}
\end{lemma}
\begin{proof}
Let $R_1=1$, $0<P_1<R_1$ arbitrary and $\rho_1<\varepsilon$ sufficiently small so that $f(\rho_1)<\in\{P_1,R_1-P_1\}$.
For $i>1$ we define inductively:
\begin{equation*}
R_{i+1}=P_i-f(\rho_i),\, P_{i+1}<\inf\{g(\rho_i),R_{i+1}\}
\end{equation*}
and $\rho_{i+1}<\inf{\rho_i,\frac{1}{i}}$ sufficiently small so that
\begin{equation*}
f(\rho_{i+1})<\inf\{P_{i+1},R_{i+1}-P_{i+1}\}.
\end{equation*}
Finally, we define 
\begin{equation*}
S(\rho)=\{P\in\real:\text{ for some } i,\, P=P_i\text{ and }\rho_i\leq\rho\}.
\end{equation*}
It is clean that $(1)$ holds.

To prove $(2)$, suppose $0<\rho<\infty$ and let $i=\min\{i:\rho_i\leq\rho\}$.
If $P_j\in S(\rho)$ then either $j=i$, in which case $P_j=P_i$ and $P_j\in\mathbf{B}(P_i,g(\rho))$ is trivial, or $j>i$ and $0<P_j\leq P_{i+1}\leq g(\rho_i)\leq g(\rho)$ hence $P_j\in\mathbf{B}(0,g(\rho))$.

Now, let $A\subset\real$ be as in $(3)$.
For all positive integer $i$ and $P_i\in S(\rho_i)$ let $Q\in A$ be such that $|Q-P_i|\leq f(\rho_i)$.
Therefore,
\begin{equation*}
R_{i+1}\leq P_i-f(\rho_i)\leq Q\leq P_i+f(\rho_i)< R_i.
\end{equation*}
Since $R_i$ is a strictly decreseasing sequence and for all $0<\rho\leq\varepsilon$ there exists a a positive integer $i$ such that $S(\rho_i)\subset S(\rho)$, then $(3)$ is proved.
\end{proof}

It is not difficult to extend this example to planes in Euclidean space, simply choose a one parameter family of planes and apply the above example.

We remark that the conclusion of \cite{hutchinson1986}*{Lemma 3.3} is correct, in fact it is implied by Corollary \ref{partition corollary}.
Therefore there is no counter-example to their result.
The point is that the metric space argument to conclude finiteness is not sufficient.
At this step it remains crucial to utilize the integrability and density conditions.
One could use the above example to construct a varifold that either violates the integrability condition or the density condition so that the set $S_\rho$ in the proof of \cite{hutchinson1986}*{Lemma 3.3} is in fact given by the example above.
For example, in $\real^2$ we can either take a sequence of straight lines with angle in $S(\rho)$ given above, which violates the density condition, or a sine-type function with zeros prescribed by the points of $S(\rho)$, which violates the integrability condition.

\bibliographystyle{plain}
\bibliography{bibliography}

\begin{bibdiv}
\begin{biblist}

\bib{allard1972}{article}{
      author={Allard, William~K.},
       title={On the first variation of a varifold},
        date={1972},
        ISSN={0003-486X},
     journal={Ann. of Math. (2)},
      volume={95},
       pages={417\ndash 491},
      review={\MR{0307015 (46 \#6136)}},
}

\bib{almgren2000}{book}{
      author={Almgren, Frederick~J., Jr.},
       title={Almgren's big regularity paper},
      series={World Scientific Monograph Series in Mathematics},
   publisher={World Scientific Publishing Co., Inc., River Edge, NJ},
        date={2000},
      volume={1},
        ISBN={981-02-4108-9},
        note={$Q$-valued functions minimizing Dirichlet's integral and the
  regularity of area-minimizing rectifiable currents up to codimension 2, With
  a preface by Jean E. Taylor and Vladimir Scheffer},
      review={\MR{1777737}},
}

\bib{bellettini-mugnai2004}{article}{
      author={Bellettini, G.},
      author={Mugnai, L.},
       title={Characterization and representation of the lower semicontinuous
  envelope of the elastica functional},
        date={2004},
        ISSN={0294-1449},
     journal={Ann. Inst. H. Poincar\'{e} C Anal. Non Lin\'{e}aire},
      volume={21},
      number={6},
       pages={839\ndash 880},
         url={https://doi.org/10.1016/j.anihpc.2004.01.001},
      review={\MR{2097034}},
}

\bib{buet-leonardi-masnou2022}{article}{
      author={Buet, Blanche},
      author={Leonardi, Gian~Paolo},
      author={Masnou, Simon},
       title={Weak and approximate curvatures of a measure: a varifold
  perspective},
        date={2022},
        ISSN={0362-546X},
     journal={Nonlinear Anal.},
      volume={222},
       pages={Paper No. 112983, 34},
         url={https://doi.org/10.1016/j.na.2022.112983},
      review={\MR{4432350}},
}

\bib{federer1969}{book}{
      author={Federer, Herbert},
       title={Geometric measure theory},
      series={Die Grundlehren der mathematischen Wissenschaften, Band 153},
   publisher={Springer-Verlag New York Inc., New York},
        date={1969},
      review={\MR{0257325 (41 \#1976)}},
}

\bib{hutchinson1986.2}{incollection}{
      author={Hutchinson, John~E.},
       title={{$C^{1,\alpha}$} multiple function regularity and tangent cone
  behaviour for varifolds with second fundamental form in {$L^p$}},
        date={1986},
   booktitle={Geometric measure theory and the calculus of variations
  ({A}rcata, {C}alif., 1984)},
      series={Proc. Sympos. Pure Math.},
      volume={44},
   publisher={Amer. Math. Soc., Providence, RI},
       pages={281\ndash 306},
         url={https://doi.org/10.1090/pspum/044/840281},
      review={\MR{840281}},
}

\bib{hutchinson1986}{article}{
      author={Hutchinson, John~E.},
       title={Second fundamental form for varifolds and the existence of
  surfaces minimising curvature},
        date={1986},
        ISSN={0022-2518},
     journal={Indiana Univ. Math. J.},
      volume={35},
      number={1},
       pages={45\ndash 71},
         url={https://doi.org/10.1512/iumj.1986.35.35003},
      review={\MR{825628}},
}

\bib{menne2010}{article}{
      author={Menne, Ulrich},
       title={A {S}obolev {P}oincar\'{e} type inequality for integral
  varifolds},
        date={2010},
        ISSN={0944-2669},
     journal={Calc. Var. Partial Differential Equations},
      volume={38},
      number={3-4},
       pages={369\ndash 408},
         url={https://doi.org/10.1007/s00526-009-0291-9},
      review={\MR{2647125}},
}

\bib{menne2013.1}{article}{
      author={Menne, Ulrich},
       title={Second order rectifiability of integral varifolds of locally
  bounded first variation},
        date={2013},
        ISSN={1050-6926},
     journal={J. Geom. Anal.},
      volume={23},
      number={2},
       pages={709\ndash 763},
         url={https://doi.org/10.1007/s12220-011-9261-5},
      review={\MR{3023856}},
}

\bib{menne2016.1}{article}{
      author={Menne, Ulrich},
       title={Weakly differentiable functions on varifolds},
        date={2016},
        ISSN={0022-2518},
     journal={Indiana Univ. Math. J.},
      volume={65},
      number={3},
       pages={977\ndash 1088},
         url={https://doi.org/10.1512/iumj.2016.65.5829},
      review={\MR{3528825}},
}

\bib{menne-scharrer2022:arxiv}{article}{
      author={Menne, Ulrich},
      author={Scharrer, Christian},
       title={{A priori bounds for geodesic diameter. Part II. Fine
  connectedness properties of varifolds}},
        date={2022},
     journal={arXiv:2209.05955v2 [math.DG]},
      eprint={2209.05955v2},
}

\bib{mondino2014}{article}{
      author={Mondino, Andrea},
       title={Existence of integral {$m$}-varifolds minimizing {$\int|A|^p$}
  and {$\int|H|^p,\,p>m,$} in {R}iemannian manifolds},
        date={2014},
        ISSN={0944-2669},
     journal={Calc. Var. Partial Differential Equations},
      volume={49},
      number={1-2},
       pages={431\ndash 470},
         url={https://doi.org/10.1007/s00526-012-0588-y},
      review={\MR{3148123}},
}

\bib{mugnai2013}{article}{
      author={Mugnai, Luca},
       title={Gamma-convergence results for phase-field approximations of the
  2{D}-{E}uler elastica functional},
        date={2013},
        ISSN={1292-8119},
     journal={ESAIM Control Optim. Calc. Var.},
      volume={19},
      number={3},
       pages={740\ndash 753},
         url={https://doi.org/10.1051/cocv/2012031},
      review={\MR{3092360}},
}

\bib{schatzle2010}{article}{
      author={Sch\"{a}tzle, Reiner},
       title={The {W}illmore boundary problem},
        date={2010},
        ISSN={0944-2669},
     journal={Calc. Var. Partial Differential Equations},
      volume={37},
      number={3-4},
       pages={275\ndash 302},
         url={https://doi.org/10.1007/s00526-009-0244-3},
      review={\MR{2592972}},
}

\bib{simon1983:gmtbook}{book}{
      author={Simon, Leon},
       title={Lectures on geometric measure theory},
      series={Proceedings of the Centre for Mathematical Analysis, Australian
  National University},
   publisher={Australian National University, Centre for Mathematical Analysis,
  Canberra},
        date={1983},
      volume={3},
        ISBN={0-86784-429-9},
      review={\MR{756417 (87a:49001)}},
}

\bib{wickramasekera2005}{article}{
      author={Wickramasekera, Neshan},
       title={On the singularities and a {B}ernstein property of immersed
  stable minimal hypersurfaces},
        date={2005},
        ISSN={0944-2669},
     journal={Calc. Var. Partial Differential Equations},
      volume={22},
      number={1},
       pages={1\ndash 20},
         url={https://doi.org/10.1007/s00526-004-0263-z},
      review={\MR{2105966}},
}

\end{biblist}
\end{bibdiv}

\end{document}